\documentclass[preprint,12pt,sort&compress]{amsart}

\usepackage{amsmath,amssymb,latexsym,amscd,amsthm}
\usepackage[dvips]{epsfig, color}
\usepackage[margin=1.18in]{geometry} 

\newtheorem{theorem}{Theorem}[section]  

\newtheorem{property}[theorem]{Property}
\newtheorem{lemma}[theorem]{Lemma}
\theoremstyle{definition}

\newtheorem{question}[theorem]{Question}

\theoremstyle{remark}

\def\bara{B\'ar\'any}

\def\cara{Carath\'eodory}

\def\R{\mathbb{R}}
\def\S{\mathbf{S}}

\def\O{{\bf \Omega}}

\def\hati{\widehat{i}}

\def\conv{\operatorname{conv}}

\def\zero{{\bf 0}}

\title{Computational Lower Bounds for Colourful Simplicial Depth}

\author{Antoine Deza}
\address{Advanced Optimization Laboratory, Department of Computing and Software,
     McMaster University, Hamilton, Ontario, Canada}
\email{deza@mcmaster.ca}
\author{Tamon Stephen}
\address{Department of Mathematics, Simon Fraser University,
British Columbia, Canada}
\email{tamon@sfu.ca}
\author{Feng Xie}
\address{Advanced Optimization Laboratory, Department of Computing and Software,
     McMaster University, Hamilton, Ontario, Canada}
\email{groovyfeng@gmail.com}

\begin{document}
\thispagestyle{empty}
\maketitle

\begin{abstract} 
The colourful simplicial depth problem in dimension $d$ is to
find a configuration of $(d+1)$ sets of $(d+1)$ points
such that the origin is contained in the convex hull of each
set (colour) but contained in a minimal number of \emph{colourful}
simplices generated by taking one point from each set.
A construction attaining $d^2+1$ simplices is known, and is
conjectured to be minimal.  This has been confirmed up to $d=3$,
however the best known lower bound for $d \ge 4$ is
$\lceil \frac{(d+1)^2}{2} \rceil$.  

A promising method to improve this lower bound is to look
at combinatorial \emph{octahedral systems} generated by
such configurations.  The difficulty to employing this
approach is handling the many symmetric configurations
that arise.  We propose a table of invariants which exclude
many of partial configurations,
and use this to improve the lower bound in dimension 4.
\end{abstract}


\section{Introduction}\label{se:intro}
A \emph{colourful configuration} is the union of $(d+1)$ sets, or colours,  $\S_0, \S_1, \ldots, \S_{d}$ of $(d+1)$ 
points in $\R^d$. We are interested in the \emph{colourful simplices} formed
by taking the convex hull of a set containing one point of each colour.
The \emph{colourful simplicial depth} problem is to find
a colourful configuration, with each $\S_i$ containing the origin $\zero$ in the interior
of its convex hull, minimizing the number of colourful
simplices containing $\zero$.  We denote this minimum by $\mu(d)$

Computing $\mu(d)$ can be viewed as refining
{\bara}'s Colourful {\cara} Theorem
\cite{Bar82} whose original version gives $\mu(d) \ge 1$, 
and $\mu(d) \ge d+1$ with a useful modification.
The question of computing $\mu(d)$ was studied in~\cite{DHST06}, 
which showed $\mu(2)=5$,  that $2d \le \mu(d) \le d^2+1$ for $d \ge 3$
and that $\mu(d)$ is even when $d$ is odd.
The lower bound has since been improved by~\cite{BM06}, who verified 
the conjecture for $d=3$, \cite{ST06} and~\cite{DSX11} 
which includes the current strongest bound of
$\mu(d) \ge \lceil \frac{(d+1)^2}{2} \rceil$ for $d \ge 4$. 

One motivation for colourful simplicial depth is to
establish bounds on ordinary simplicial depth.  
A point $p \in \R^d$ has {\it simplicial depth} $k$
relative to a set $S$ if it is contained in $k$ closed simplices 
generated by $(d+1)$ sets of $S$.  
This was introduced by Liu \cite{Liu90} as a statistical 
measure of how representative $p$ is of $S$. 
See \cite{Gro10,MW12,Kar12,KMS12} for recent progress on this problem.
We remark also that the colourful simplicial depth of a
point is the number of solutions to a colourful linear program
in the sense of \cite{BO97} and \cite{DHST08}.

Our strategy, following~\cite{CDSX11}, is to show that a particular 
hypergraph whose edges correspond to the colourful simplices
containing $\zero$ in a configuration cannot exist.

\subsection{Preliminaries}\label{se:prelim}
A \emph{colourful configuration} is a collection of $(d+1)$ sets 
$\S_0, \S_1, \ldots, \S_{d}$ of $(d+1)$ points in $\R^d$.
Let $\S = \cup_{i=0}^{d} \S_i$.
Without loss of generality we assume
that the points in $\S \cup \{ \zero \}$ are in general position.
Recall that $\mu(d)$ represents the minimal number of 
\emph{colourful simplices} generated by one point from each
$\S_i$ and containing $\zero$ in any $d$-dimensional colourful configuration
with $\zero\in\cap_{i=1}^{d+1}\conv(\S_i)$.

We take the simplices to be closed and remark that the
minimum should be attained.
{\bara}'s Colourful {\cara} Theorem~\cite{Bar82} gives that
some colourful simplex containing $\zero$ exists
in a colourful configuration with $\zero\in\cap_{i=1}^{d+1}\conv(\S_i)$, 
and
in fact that every point in $\S$ is part of some such
colourful simplex containing $\zero$.  This guarantees that 
$\mu(d) \ge d+1$.
A colourful configuration of~\cite{DHST06} in $\R^d$ has only $d^2+1$
colourful simplices containing $\zero$, thus $\mu(d) \le d^2+1$.
It is known that $\mu(1)=2$ (trivial), $\mu(2)=5$ \cite{DHST06}, 
$\mu(3)=10$ \cite{BM06} and that 
$\mu(d) \ge \lceil \frac{(d+1)^2}{2} \rceil$ for $d \ge 4$ \cite{DSX11}.

A colourful configuration defines a
$(d+1)$-uniform hypergraph on $\S = \displaystyle \cup_{i=1}^{d+1} \S_i$
by taking edges corresponding to the vertices of $\zero$
containing colourful simplices.  We will call a hypergraph that
arises from a colourful configuration with $\zero\in\cap_{i=1}^{d+1}\conv(\S_i)$
 a \emph{configuration hypergraph}.
The Colourful {\cara} Theorem gives that any
configuration hypergraph  must satisfy:
\begin{property}\label{prop:cover}
Every vertex of a configuration hypergraph belongs to at least one of its edges.
\end{property}

Fix a colour $i$.  We call a set $t$ of $d$ points that contains
exactly one point from each $\S_j$ other than $\S_i$ an
$\widehat{i}-transversal$.  That is to say, an 
$\widehat{i}-$transversal  $t$ has $t \cap \S_i = \emptyset$
and $|t \cap \S_j|=1$ for $i \ne j$.
We call any pair of disjoint $\widehat{i}-$transversals
an $\hati$-\emph{octahedron}; these may or may not generate
a cross-polytope ($d$-dimensional octahedron) in the 
geometric sense that their convex hull is a cross-polytope
with same coloured points never adjacent in the skeleton
of the polytope. 

A key property of colourful configurations is that for a 
fixed $\hati$-octahedron $\O$, the parity of the number of colourful
simplices containing $\zero$ formed using points from $\O$ and
a point of colour $i$ does not
depend on which point of colour $i$ is chosen.
This is a topological fact that corresponds to the fact
that $\zero$ is either ``inside'' or ``outside'' the
octahedron,
see for instance the {\em Octahedron Lemma} of \cite{BM06} for a proof.
Figure~\ref{figOcta} illustrates this in a two dimensional case
where $\zero$ is at the centre of a circle that contains points of
the three colours.
\begin{figure}[htb]
\begin{center}
\epsfig{file=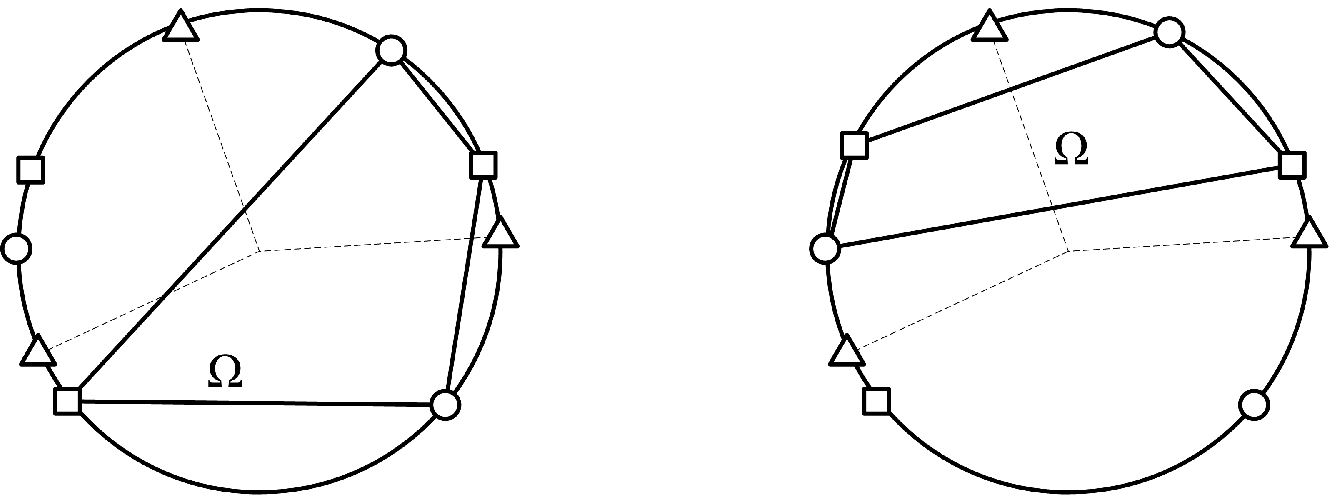, height=3.0cm}
\caption{Two-dimensional cross-polytopes $\O$ containing $\zero$ and not.} 
\label{figOcta}
\end{center}
\end{figure}

We carry the definitions of $\hati$-\emph{transversals} and
$\hati$-\emph{octahedra} over to the hypergraph setting.
Then any hypergraph arising from a colourful configuration must satisfy:
\begin{property}\label{prop:oct}
For any octahedron $\O$ of a hypergraph arising from a colourful configuration, the parity of the set of edges
using points from $\O$ and a fixed point $s_i$ for the $i${\rm th} 
coordinate is the same for all choices of $s_i$.  
\end{property}

Consider a hypergraph whose vertices are 
$\S = \cup_{i=0}^d \S_i$
and whose edges have exactly one element from each set.
If the hypergraph satisfies Property~\ref{prop:oct} we call it an 
\emph{octahedral system}, if it additionally satisfies
Property~\ref{prop:cover} we call it a 
\emph{octahedral system without isolated vertex}.  A
colourful configuration with $\zero\in\cap_{i=1}^{d+1}\conv(\S_i)$ and $k$ colourful simplices 
containing $\zero$ has a configuration hypergraph that is
an octahedral system without isolated vertex
with $k$ edges.  Then any lower bound for the number of
edges in an octahedral system without isolated vertex with $(d+1)$ colours is also 
a lower bound $\nu(d)$ for $\mu(d)$.  It is an interesting question
whether there are any octahedral systems without isolated vertex not arising
from any colourful configurations, and if not whether $\nu(d) < \mu(d)$
for some $d$.  This purely combinatorial approach was
originally suggested by {\bara}.

\section{A table of invariants} \label{se:table}
Octahedral systems have the advantage of being combinatorial
and finite.  In principle for any particular $d$ and $k$ we
can check if there exists an octahedral system without isolated vertex on 
$\S = \cup_{i=0}^d \S_i$ with up to
$k$ edges by generating all the (finitely many) hypergraphs
with up to $k$ edges, each containing one element from 
each $\S_i$ and then testing if they satisfy
Properties~\ref{prop:cover} and~\ref{prop:oct}.
The difficulty, of course, lies in the sheer number
of such hypergraphs, and in verifying
Property~\ref{prop:oct} efficiently.  

We can reduce the search space by exploiting
the many combinatorial symmetries in such hypergraphs 
and considering only configurations that satisfy certain
normalizations.  However, this alone is not sufficient
to improve the known lower bounds even for $d=4$.  We thus turn our
attention to how to use Property~\ref{prop:oct} effectively.
Since a given configuration has very many octahedra,
in fact $(d+1) {d \choose 2}^d$, we expect that most 
unstructured hypergraphs fail Property~\ref{prop:oct}
for many octahedra.  

There are so many octahedra that it would be
difficult to verify Property~\ref{prop:oct} explicitly for 
non-trivial octahedral systems.  However, we can often
quickly detect violations of Property~\ref{prop:oct}
since hypergraphs which are not octahedral systems
may fail Property~\ref{prop:oct} for most octahedra.
Even more, we will show that certain partial systems 
cannot satisfy Property~\ref{prop:oct} even with the 
addition of a number of edges.

\subsection{The large table} \label{se:lt}
We begin by fixing an arbitrary colour as colour 0
and an arbitrary $0$-transversal.  We can label the
points in each set from $0$ to $d$ and without loss
of generality take the transversal to contain the
0 point of each set.  For convenience we write edges
as a string of $(d+1)$ numbers and transversals as
string of $d$ numbers with $\ast$ corresponding to the
omitted colour.  Thus the $0$-transversal considered
is $\ast 0 0 \ldots 0$.  

There are $d^d$ possible octahedra formed by choosing
a second $0$-transversal disjoint from $\ast 0 0 \ldots 0$.
For a hypergraph of colourful edges on $\S$,
we generate a $d^d \times (d+1)$ table whose rows
are indexed by the octahedra containing $\ast 0 0 \ldots 0$
and whose columns are indexed by the points of colour 0.
The entries of the table are the parity of the number of
edges using points from the octahedron corresponding to
the row and the point of colour 0 corresponding to the
column.  We call this the \emph{large table}.
We observe that Property~\ref{prop:oct} implies that
if the hypergraph is an octahedral system, the rows of the large 
table must be constant -- either all zeros or all ones.

For any hypergraph if we build the large table, by taking
the view that the entries in the 0 column are
correct we can produce a \emph{score} for the hypergraph by
simply counting the entries in the remaining columns
that do not agree with the 0 column entry.  
So any octahedral system will
have a score of 0, while the hypergraph that consists
only of the edge $000 \ldots 0$ has the maximum possible
score of $d^{d+1}$.  Let $z(e)$ be the number of zeros
in hypergraph edge $e$.  Then we have:

\begin{lemma} \label{le:edge}
Adding or deleting edge $e$ to an octahedral system changes
its score by at most $d^{z(e)}$.
\end{lemma}

\begin{proof}
Consider first the case where the element of colour 0 of $e$ is
$j \ne 0$.  Then the only entries of the large table
that change are those in column $j$.  Further, an
entry will change if and only if the remaining elements
of $e$ lie in the octahedron corresponding to the row.
Then the affected rows are those that whose non-zero
transversal agrees with the non-zero elements of $e$.
For zero elements of $e$ any choice of transversal element
is allowed.  This gives the requisite number of affected
entries.

In the case where the 0th element of $e$ is 0, 
elements of column 0 are changed, thus changing the correctness
of the remaining $d$ entries in the affected rows,
which themselves do not change.
The number of rows affected again depends on the number
of zero elements in the remainder of $e$. This is one
less than the total number of zero elements in $e$,
again giving the requisite number of affected entries.
\end{proof}

Lemma~\ref{le:edge} allows us to make the following
useful observation.  Call an edge $e$ of a hypergraph
 \emph{isolated} if there is no other edge
that differs from $e$ only in a single coordinate.
Then:

\begin{lemma} \label{le:no_iso}
An octahedral system with $d^2$ or fewer edges must
not contain any isolated edges.
\end{lemma}

\begin{proof}
Without loss of generality we take the isolated edge
to be $0 0 0 \ldots 0$ and consider the large table
formed.  The hypergraph consisting of the unique edge
$0 0 0 \ldots 0$ has score $d^{d+1}$.  Now keep track
of the score as additional edges are added one at a time.
Since $0 0 0 \ldots 0$ is isolated, all subsequent edges
have at most $d-1$ zeros, and thus reduce the score by
at most $d^{d-1}$.  So at least $\frac{d^{d+1}}{d^{d-1}}=d^2$
additional edges are required to reduce the score to zero.
The claim follows.
\end{proof}

\subsection{The small table} \label{se:st}
For some cases that we are interested in it is possible
to compute the score for the large table in reasonable
time, but for the purpose of quickly showing that
a given hypergraph is not an octahedral system (and may be
far from one), it is effective to initially focus on
a subset of the rows.  One such small subset are the
$d$ rows generated by transversals $\ast 1 1 \ldots 1$,
$\ast 2 2 \ldots 2$, $\ldots$, $\ast d d \ldots d$.
We call the restriction of the large table to these $d$ rows
the \emph{small table}.  Note that the initial numberings
are arbitrary, and the composition of the small table 
depends on this numbering, which we may fix later as
part of our search algorithm.

The advantage of focusing on this $d \times (d+1)$ table
is that the entries are relatively independent.  Only
edges of the form $x 0 0 \ldots 0$ can change more
than one entry of this table.  After accounting for such
edges, each entry can be only be affected by the $2^d-1$ edges
that are on the relevant octahedron with the given final
coordinate.

\section{Enumeration Strategy and Computation}
In this section we describe how the small and large
tables, combined with symmetry breaking strategies, 
can be used to search efficiently for small 
octahedral systems.  This strategy was implemented \cite{Xie12}
in Python version 2.6 on an AMD Opteron Processor 8356 
core (2.3G Hz) and is able to prove that $\mu(4) \ge 14$ in about
30 days of CPU time.
This improves by 1 the bound of \cite{DSX11}.

Our strategy is enumerative: we try to build an octahedral
system without isolated vertex by adding one edge at a time.  At the start we 
reduce our choices by using symmetry.  We then seek to
add edges that are required by the small table.
Subsequently we seek to add edges that are required by
Lemma~\ref{le:no_iso}, and only as a last resort do we
consider arbitrary edges.  As we branch, we attempt to quickly identify
partial configurations that cannot extend to a sufficiently small
octahedral system.

\subsection{Initial assumptions} \label{se:init}
We begin by fixing an initial colour 0 and transversal
$t_0 = \ast 0 0 \ldots 0$;  the tables we consider are with
respect to $t_0$.  We then use the results
of \cite{DSX11}, to break the problem into several
cases based on $l$, the number of edges containing $t_0$,
$b$ the number of the small table octahedra that have
odd parity, and $j$ the minimum number of transversals
covering any point of colour 0.  
Recall that the small table octahedra are
those formed by $t_0$ with each of $t_i = \ast i i i \ldots i$
for $i=1,2, \ldots, d$.

It is clear that for any octahedral
system without isolated vertex and
with $d^2$ or fewer edges we must have 
$1 \le l,b,j \le d$ and that the number of edges is
at least $j (d+1)$.  Further, \cite{DSX11} shows that
we must have $j+b \ge d+1$, and that the number
of edges must be at least $(b+l)(d+1)-2bl$, as well as
at least $dl + 1$ assuming that the colour 0 is chosen to minimize
$l$ and that $l \ge \frac{d+2}{2}$.
This last fact allows us to assume that $l \le d-1$ if
we have an octahedral
system without isolated system with less than
$d^2+1$ edges.

\subsection{Details of enumeration} \label{se:14}
In the following we describe an enumeration that
excludes $\mu(4) = 13$.  This improves the bound of \cite{DSX11} by 1.  
To rule out possible octahedral systems without isolated vertex of size 13 (or 14),
it is sufficient to consider cases where $j=1$ or $j=2$.
which in turn means $b=3$ or $b=4$.  In the case $b=3$,
we have at least $15 - l$ simplices, so $l=2$ or $l=3$,
and in the case $b=4$, we have $20 - 3l$ so $l=3$; we would need to
consider additional cases for $l$ and $b$ to rule out 14.
In summary, we need to rule out systems where the triple
$(l,b,j)$ is one of $(3,4,2), (3,4,1), (3,3,2), (2,3,2)$.

By reordering the points of colour 0, we can take the
edges $x 0 0 0 0$ to be in the system for $0 \le x \le l-1$,
and not in the system for $l \le x \le 4$.  Consider
the small table after including these edges with $l=3$,
illustrated in Table~\ref{ta:l3}.
\begin{table}[h!]
  \begin{tabular}{|c|ccccc|}
    \hline
    & 0 & 1 & 2 & 3 & 4 \\
    \hline
    $\ast 1 1 1 1$ & 1 & 1 & 1 & 0 & 0 \\
    $\ast 2 2 2 2$ & 1 & 1 & 1 & 0 & 0 \\
    $\ast 3 3 3 3$ & 1 & 1 & 1 & 0 & 0 \\
    $\ast 4 4 4 4$ & 1 & 1 & 1 & 0 & 0 \\
    \hline
  \end{tabular}
  \caption{The small table with $l=3$}\label{ta:l3}
\end{table}
Now if $b=4$, then we are requiring that the small table
be comprised entirely of 1's.  So in this case the entries
in the first 3 columns are correct, while the entries in
the last 2 columns are incorrect.  

For $(l,b,j)=(3,4,2)$ we proceed to enumerate configurations as follows.
Since $l=3$, we include initial edges $00000, 10000, 20000$.
We then add edges to correct each of the 8 entries of Table~\ref{ta:l3}
which must be fixed to get the correct small table
for $b=4$.  As previously remarked, adding any edge not
of the form $x0000$ will change only a single entry in
the small table.  For instance, the entry in the first row
and fourth column can be changed only by an edge of the
form $3 a b c d$ where $a,b,c,d \in \{0,1\}$.  Given that
that $30000$ cannot be added to the configuration without changing $l$, there
remain only 15 possible edges that change the entry,
and one must be in our configuration.  In fact, by reordering
the colours we can take it to be one of $31000$, $31100$, $31110$
and $31111$.  

We could continue to exploit symmetries in this way -- for instance
depending on which of the previous 4 edges is chosen, 
the next edge could be one of 4 to 7 edges fixing the next
table entry.  However, we did not do this so as to avoid
extensive case analysis.  Instead, we began branching on
all 15 possible edges that switch a given table entry until
the table is correct and the partial configuration has 11
edges.

As we branch we check
two simple predictors that may indicate that the configuration 
requires several more edges.  First, we look for points
that are not currently included in any edge.
If some colour still has $k$ uncovered points, then we
require $k$ additional edges.  Second, since any vertex of
colour 0 must be covered by at least $j$ edges, we look
to see which points of colour 0 are not contained in 
sufficiently many edges, and get a score $k'$ by totaling
the undercounts.  At the same time, we may find that all
vertices of colour 0 are already covered by more than $j$
edges (especially when $j=1$), in which case the partial
configuration no longer belongs to this subcase and can
be excluded. 
Again, we require $k'$ additional edges.
If either $k$ or $k'$ is sufficiently large (in this case 3),
then the current partial configuration cannot extend to an
octahedral system without isolated vertex with less than 14 edges and is
abandoned.

Otherwise, we examine the configuration to see if it
has an isolated edge.  If it contains an isolated edge $e$, 
the by Lemma~\ref{le:no_iso}, if the configuration is
to extend to an octahedral system without isolated vertex with less than
17 edges, it must include an edge adjacent to $e$.
That is, it must contain $e'$ differing from $e$ only
in a single coordinate.  There are only 20 such edges
so we can branch on them.  We then repeat the process
of applying predictors and looking for an isolated
edge until we either find an octahedral system without isolated vertex
with less than 14 edges, or all partial configurations
with fewer edges are exhausted.

If we do arrive at a partial configuration with no
isolated edges, then as a last resort we may have to
branch on all possible edges.
However, this happens infrequently enough that the enumeration ends in a
reasonable time.  

The remaining cases, where $(l,b,j)$ is $(3,4,1), (3,3,2)$ or
$(2,3,2)$ are similar.  
Having exhausted all these cases, we conclude that $\mu(4) \ge 14$.

\section{Conclusions and remarks}\label{se:conclusions}
Octahedral systems are intriguing combinatorial objects.
Using the observation that configuration hypergraphs generate 
octahedral systems without isolated vertex, we propose a computational approach to 
establishing lower bounds for colourful simplicial depth.
A straightforward implementation of this improves the best
known lower for 4-dimensional configurations from $\mu(4) \ge 13$
to $\mu(4) \ge 14$.

We can ask several other questions about octahedral systems.
We remark that the maximum cardinality octahedral system without isolated vertex 
is the set of all possible edges; if we have $(d+1)$ sets of 
cardinality $(d+1)$ it has size $(d+1)^{d+1}$.  It is the
hypergraph arising from the colourful configuration of points in $\R^{d}$
that places the sets $\S_1, \ldots, \S_{d+1}$ close to
vertices $v_1, \ldots v_{d+1}$ respectively of a regular
simplex containing $\zero$.
\begin{question}
Can all octahedral systems without isolated vertex 
on $(d+1)$ sets of $(d+1)$ points arise from  colourful configurations
in $\R^{d}$? 
\end{question}

We conclude by mentioning that many aspects of colourful
simplices are just beginning to be explored.  For instance,
the combinatorial complexity of a system of colour simplices
is analyzed in \cite{ST12}.  As far as we know the algorithmic
question of computing colourful simplicial depth is untouched, 
even for $d=2$ where several interesting algorithms for
computing the monochrome simplicial depth have been developed.
See for instance the survey \cite{Alo06}.

\section{Acknowledgments}\label{se:ack}
This work was supported by grants from the 
Natural Sciences and Engineering Research Council of Canada (NSERC)
and MITACS, and by the Canada Research Chairs program.
T.S.~thanks the University of Cantabria for hospitality
while working on this paper.

\bibliographystyle{amsalpha}

\bibliography{refs}

\providecommand{\bysame}{\leavevmode\hbox to3em{\hrulefill}\thinspace}
\providecommand{\MR}{\relax\ifhmode\unskip\space\fi MR }
\providecommand{\MRhref}[2]{%
  \href{http://www.ams.org/mathscinet-getitem?mr=#1}{#2}
}
\providecommand{\href}[2]{#2}
\begin{thebibliography}{DHST08}

\bibitem[Alo06]{Alo06}
Greg Aloupis, \emph{Geometric measures of data depth}, Data depth: robust
  multivariate analysis, computational geometry and applications, DIMACS Ser.
  Discrete Math. Theoret. Comput. Sci., vol.~72, Amer. Math. Soc., Providence,
  RI, 2006, pp.~147--158.

\bibitem[B{\'a}r82]{Bar82}
Imre B{\'a}r{\'a}ny, \emph{A generalization of {C}arath\'eodory's theorem},
  Discrete Mathematics \textbf{40} (1982), no.~2-3, 141--152.

\bibitem[BM07]{BM06}
Imre B{\'a}r{\'a}ny and Ji{\v r}{\' i} Matou{\v s}ek, \emph{Quadratically many
  colorful simplices}, SIAM Journal on Discrete Mathematics \textbf{21} (2007),
  no.~1, 191--198.

\bibitem[BO97]{BO97}
Imre B{\'a}r{\'a}ny and Shmuel Onn, \emph{Colourful linear programming and its
  relatives}, Math. Oper. Res. \textbf{22} (1997), no.~3, 550--567.

\bibitem[CDSX11]{CDSX11}
Grant Custard, Antoine Deza, Tamon Stephen, and Feng Xie, \emph{Small
  octahedral systems}, Proceedings of the 23rd Annual Canadian Conference on
  Computational Geometry, Toronto, Ontario, Canada, 2011, pp.~267--272.

\bibitem[DHST06]{DHST06}
Antoine Deza, Sui Huang, Tamon Stephen, and Tam{\' a}s Terlaky, \emph{Colourful
  simplicial depth}, Discrete and Comput.~Geom. \textbf{35} (2006), no.~4,
  597--604.

\bibitem[DHST08]{DHST08}
\bysame, \emph{The colourful feasibility problem}, Discrete Appl. Math.
  \textbf{156} (2008), no.~11, 2166--2177.

\bibitem[DSX11]{DSX11}
Antoine Deza, Tamon Stephen, and Feng Xie, \emph{More colourful simplices},
  Discrete Comput.~Geom. \textbf{45} (2011), no.~2, 272--278.

\bibitem[Gro10]{Gro10}
Mikhail Gromov, \emph{Singularities, expanders and topology of maps. part 2:
  from combinatorics to topology via algebraic isoperimetry}, Geom. Funct.
  Anal. \textbf{20} (2010), no.~2, 416--526.

\bibitem[Kar12]{Kar12}
Roman Karasev, \emph{A simpler proof of the
  {B}oros-{F}{\"u}redi-{B}{\'a}r{\'a}ny-{P}ach-{G}romov theorem}, Discrete
  Comput.~Geom. \textbf{47} (2012), no.~3, 492--495.

\bibitem[KMS12]{KMS12}
Daniel Kr{\'a}l', Luk{\'a}{\v{s}} Mach, and Jean-S{\'e}bastien Sereni, \emph{A
  {N}ew {L}ower {B}ound {B}ased on {G}romov's {M}ethod of {S}electing {H}eavily
  {C}overed {P}oints}, Discrete Comput. Geom. \textbf{48} (2012), no.~2,
  487--498.

\bibitem[Liu90]{Liu90}
Regina~Y. Liu, \emph{On a notion of data depth based on random simplices}, Ann.
  Statist. \textbf{18} (1990), no.~1, 405--414.

\bibitem[MW12]{MW12}
Ji{\v r}{\' i} Matou{\v s}ek and Uli Wagner, \emph{On {G}romov's method of
  selecting heavily covered points}, {\tt arXiv:1102.3515}, 2012.

\bibitem[ST08]{ST06}
Tamon Stephen and Hugh Thomas, \emph{A quadratic lower bound for colourful
  simplicial depth}, J. Comb. Opt. \textbf{16} (2008), no.~4, 324--327.

\bibitem[ST12]{ST12}
Andr{\'e} Schulz and Csaba~D. T{\'o}th, \emph{The union of colorful simplices
  spanned by a colored point set}, Computational Geometry (2012), to appear.

\bibitem[Xie12]{Xie12}
Feng Xie, \emph{{Python} code for octrahedral system computation}, available
  at:\\ http://optlab.mcmaster.ca/om/csd/, 2012.

\end{thebibliography}

\end{document}